\documentclass[11pt,a4paper]{article}
\usepackage{booktabs,tabularx,array}
\usepackage{algorithm,algpseudocode,placeins}
\usepackage{enumitem,booktabs,tabularx,array}
\usepackage[T1]{fontenc}
\usepackage{lmodern}
\usepackage[margin=27mm,headheight=14pt]{geometry}
\usepackage{amsmath,amssymb,amsthm,mathtools,bm}
\usepackage{microtype}
\usepackage{enumitem,booktabs,tabularx,array}
\usepackage[numbers,sort&compress]{natbib}
\usepackage{xurl}
\usepackage{xcolor}
\definecolor{linkblue}{RGB}{26,69,109}
\usepackage[colorlinks=true,linkcolor=linkblue,citecolor=linkblue,urlcolor=linkblue,
 pdfauthor={},pdftitle={Revisiting AdaGrad in Stochastic Convex Optimization: Last Iterates, High Probability, and Lower Bounds}]{hyperref}
\usepackage{fancyhdr}
\setlist{itemsep=3pt,topsep=5pt}
\numberwithin{equation}{section}
\newtheorem{theorem}{Theorem}[section]

\theoremstyle{definition}

\theoremstyle{remark}
\newtheorem{remark}[theorem]{Remark}
\theoremstyle{remark*}
\newtheorem{remark*}[theorem]{Remark}
\theoremstyle{assumption}
\newtheorem{assumption}{Assumption}

\newcommand{\polylog}{\operatorname{polylog}}

\title{\textbf{Revisiting AdaGrad in Stochastic Convex Optimization: Last Iterates, High Probability, and Lower Bounds}}
\author{Weiming Ou\footnote{Shanghai University of Finance and Economics, email: ouweiming@163.sufe.edu.cn}\qquad Xiao Wang\footnote{MoE Key Laboratory of Interdisciplinary Research of Computation and Economics, Shanghai University of Finance and Economics, email: wangxiao@sufe.edu.cn}}
\begin{document}\maketitle
\begin{abstract}
AdaGrad and AdaGrad-Norm are widely used adaptive methods, but their precise behavior in stochastic convex optimization remains less understood.  We first show that AdaGrad-Norm and AdaGrad do not admit any universal \textbf{last-iterate} rate, even under sub-Gaussian noise and bounded iterates.  We then prove that bounded variance alone is too weak: even with bounded iterates, it cannot yield \textbf{high-probability average-iterate} rates, and without bounded iterates it may not even guarantee convergence in expectation. Besides, we construct tight $\Omega(\log T/\sqrt T)$ lower bounds for both AdaGrad-Norm and AdaGrad under sub-Gaussian noise, showing that $\log T$ in existing average-iterate upper bounds is unavoidable.  Finally, we show that this $\log T$ loss disappears once bounded-iterate condition is imposed: under a general ABC condition, both methods achieve rates ${O}(1/\sqrt{T})$.
\end{abstract}
\clearpage\tableofcontents\clearpage

\section{Introduction}
AdaGrad scales its updates using the gradients observed so far, replacing a prescribed stepsize schedule by a data-dependent normalization. The original coordinate-wise method was developed in online convex optimization by \citet{mcmahan2010adaptive} and \citet{duchi2011adaptive}. AdaGrad-Norm instead uses a single accumulator for the cumulative squared gradient norms \citep{ward2020adagrad}. Although the methods share a common regret-analysis framework, their stochastic guarantees depend on the noise model, the output, and the behavior of the iterates.

For smooth convex objectives, average-iterate bounds are available for delayed AdaGrad-type schemes \citep{li2019convergence}, standard AdaGrad-Norm under suitable tail assumptions \citep{attia2023sgd,liu2023convergence}, and coordinate-wise AdaGrad on bounded domains \citep{liu2025adagrad}. Log-free, noise-adaptive rates also have precedents on bounded domains or under bounded-iterate assumptions \citep{levy2018online,vaswani2020adaptive}. These guarantees do not automatically describe the final point returned by an implementation. Our central question is therefore whether the standard updates admit an explicit, vanishing \emph{last-iterate} rate under the same smoothness, light-tail, and stability conditions.

Our main result gives a negative answer in the general conditional-oracle model. For every proposed positive sequence $r_T\to0$, we construct one smooth convex objective and one infinite-horizon, conditionally unbiased sub-Gaussian oracle such that the expected last-iterate error exceeds $r_T$ by an unbounded factor along a deterministic subsequence. All iterates remain inside an arbitrarily prescribed radius. Thus the issue is not merely a missing polynomial exponent: \emph{no prescribed vanishing rate is valid uniformly over this class}. The hard instance itself converges, separating convergence from a uniform quantitative guarantee. Its oracle is time- and history-dependent, so the result does not preclude rates under additional sampling restrictions.

A second question concerns the confidence of an \emph{averaged} output. Averaging often gives guarantees unavailable for the last iterate, but we show that it does not repair the poor confidence dependence of unmodified AdaGrad under finite variance. Our construction combines three properties: a \emph{fixed infinite-horizon instance}, the \emph{uniformly averaged iterate}, and \emph{any prescribed almost-sure trajectory radius}. This output distinction matters: the lower bound in \citet[Theorem~B.4]{chezhegov2025clipping}, specialized to zero momentum, is stated for the last iterate. A last-iterate lower bound alone does not establish an average-iterate lower bound. Our result excludes polynomial decay in time with only polylogarithmic dependence on the inverse failure probability.

The remaining results place these two limitations in context. Without a common stability radius, finite variance permits constant worst-case expected average error. Even under bounded, hence sub-Gaussian, noise, a scalar construction gives an $\Omega(\log T/\sqrt T)$ average-error lower bound for the fixed-parameter unprojected updates and a corresponding inverse-square-root SGD schedule. With bounded iterates, by contrast, classical regret analysis yields explicit noise-adaptive average-error bounds under the conditional ABC model. The distinction between a fixed instance and a horizon-dependent family is kept explicit throughout: choosing a different function for each horizon does not establish nonconvergence on one fixed instance.

\subsection{Our contributions}
Table~\ref{tab:summary} summarizes the results in the order developed below. Every lower-bound construction is one-dimensional and applies to both AdaGrad-Norm and coordinate-wise AdaGrad.
\begin{enumerate}[leftmargin=*,itemsep=4pt,topsep=3pt]
\item \textbf{No universal vanishing last-iterate rate.}
Theorem~\ref{thm:no last-iterate rate} shows that every prescribed $r_T\to0$ is violated by a single fixed sub-Gaussian instance with bounded iterates. Its expected error satisfies $\mathbb E[\Delta_{T_n}]/r_{T_n}\to\infty$ along a deterministic subsequence. Delayed bounded pulses keep the accumulator small until an arbitrarily late time; the resulting displacement defeats the proposed rate without preventing convergence of the constructed instance.
\item \textbf{A high-probability limitation for the average iterate.}
Theorem~\ref{thm:hp} excludes polynomial-in-$T$ rates with only polynomial dependence on $\log(1/\delta)$ for \emph{uniform averaging}, on one fixed infinite-horizon finite-variance instance inside any prescribed radius $D>0$. A pulse cascade increases the accumulator through alternating small displacements. The combination of fixed-instance quantifiers, averaged output, and arbitrary stability radius distinguishes the statement from the last-iterate lower bound of \citet{chezhegov2025clipping}; the general heavy-tail difficulty itself is already known.
\item \textbf{Average-error lower bounds without a uniform stability radius.}
Theorem~\ref{thm:no-bounded} gives a horizon-dependent family with constant worst-case expected average error under independent identically distributed additive finite-variance noise. Theorem~\ref{thm:log-lb} gives $\Omega(\eta\sigma\log T/\sqrt T)$ under sub-Gaussian noise, even from a minimizer. Both concern fixed algorithm parameters. The latter construction also proves a smooth-convex average-error lower bound for inverse-square-root SGD; the final part of Appendix~\ref{app:log-lb} shows that its SGD version can use independent bounded additive noise.
\item \textbf{An explicit noise-adaptive ABC upper bound.}
Under bounded iterates, Theorem~\ref{thm:norm upper bound} gives $O(T^{-1}+\sqrt{C/T})$ expected average error, with initialization and the ABC coefficients explicit. The $C=0$ case gives $O(T^{-1})$ even with state-dependent noise. This is a unified extension of classical regret and self-bounding analysis \citep{duchi2011adaptive,levy2018online,vaswani2020adaptive}, not a new noise-adaptation principle. It assumes, rather than proves, trajectory stability.
\end{enumerate}

\begin{table}[t]
\centering
\small
\setlength{\tabcolsep}{3.5pt}
\renewcommand{\arraystretch}{1.15}
\begin{tabularx}{\linewidth}{@{}>{\raggedright\arraybackslash}p{0.105\linewidth}>{\raggedright\arraybackslash}p{0.27\linewidth}>{\raggedright\arraybackslash}p{0.225\linewidth}>{\raggedright\arraybackslash}X@{}}
\toprule
Theorem & Assumptions & Guarantee type & Result \\
\midrule
\ref{thm:no last-iterate rate} & Sub-Gaussian + bounded iterates & Last iterate & No universal vanishing rate\textsuperscript{a} \\
\midrule
\addlinespace[3pt]
\ref{thm:hp} & Finite variance + bounded iterates & Average iterate,\newline high probability & No polynomial rate\textsuperscript{b} with $\polylog(1/\delta)$ \\
\midrule
\addlinespace[3pt]
\ref{thm:no-bounded} & Finite variance & Average iterate & Lower bound\textsuperscript{c} $\Omega(1)$ \\
\midrule
\addlinespace[3pt]
\ref{thm:log-lb} & Sub-Gaussian & Average iterate & Lower bound\textsuperscript{c}\newline $\Omega(\log T/\sqrt T)$ \\
\midrule
\addlinespace[3pt]
\ref{thm:norm upper bound} & ABC condition + bounded iterates & Average iterate & Upper bound\newline $O(T^{-1}+\sqrt{C/T})$ \\
\bottomrule
\end{tabularx}
\caption{Results for standard unprojected AdaGrad-Norm and AdaGrad on smooth convex objectives with conditionally unbiased oracles. Errors are in expectation unless high probability is specified.
\textsuperscript{a}\,For each prescribed $r_T\to0$, one fixed infinite-horizon instance satisfies $\mathbb E[\Delta_{T_n}]/r_{T_n}\to\infty$.
\textsuperscript{b}\,One fixed infinite-horizon instance, uniformly averaged output, and any prescribed radius $D>0$; the excluded bounds have polynomial decay in $T$ and only polynomial dependence on $\log(1/\delta)$.
\textsuperscript{c}\,Horizon-dependent families with fixed algorithm parameters and no common trajectory radius. The upper bound is an explicit ABC extension of classical analysis.}
\label{tab:summary}
\end{table}

\subsection{Related work}
\paragraph{Last-iterate convergence versus uniform rates.}
Last-iterate convergence has been studied extensively for stochastic gradient methods \citep{bertsekas2000gradient,zhang2004solving,shamir2013stochastic,orabona2020almost,orabona2020unbounded,liu2022almost}, stochastic mirror descent \citep{liu2023revisiting}, and momentum methods based on follow-the-regularized-leader updates \citep{li2022last}. For adaptive methods, \citet{li2019convergence} analyze delayed updates, while \citet{jin2022convergence,jin2026longrun} study qualitative stochastic convergence of current-gradient AdaGrad under their respective assumptions. The deterministic nonsmooth analysis of \citet{preobrazhenskaia2026last} concerns another setting. Theorem~\ref{thm:no last-iterate rate} does not assert a first qualitative convergence theorem. It instead rules out \emph{every prescribed uniform vanishing rate} under smoothness, bounded iterates, and a general conditional sub-Gaussian oracle. The distinction is substantive: almost-sure convergence of individual trajectories does not supply a uniform expected rate for the entire oracle class.

\paragraph{High-probability guarantees and the averaged output.}
Foundational and subsequent high-probability analyses of stochastic mirror descent and gradient methods include \citet{nemirovski2009robust,gorbunov2020stochastic,gorbunov2021near,li2020high}. Clipping can control heavy-tailed gradients \citep{cutkosky2021high,liu2024high}; high-probability analyses for adaptive methods include \citet{kavis2022high,attia2023sgd,hong2024revisiting,hong2024convergence}. Most directly, \citet{chezhegov2025clipping} establish poor confidence dependence for unmodified adaptive methods and show how clipping improves it. Their Theorem~B.4 concerns the event $f(x_K)-f(x^\star)\ge\varepsilon$ for the \emph{last iterate} of their momentum-AdaGrad update; setting the momentum coefficient to zero recovers AdaGrad-Norm. Our Theorem~\ref{thm:hp} directly controls $f(\bar x_T)-f(x^\star)$ instead, on one fixed infinite-horizon instance within an arbitrary prescribed radius. The proof keeps the post-cascade iterates on the same side of the minimizer, so averaging cannot cancel the long-lived displacement. This is an averaged-output obstruction, not an inference from a last-iterate event. It remains compatible with weaker confidence bounds obtained from expectation by Markov's inequality.

\paragraph{Lower bounds for SGD and adaptive methods.}
Classical oracle-complexity and information-theoretic results establish an $\Omega(T^{-1/2})$ stochastic term for suitable convex problem classes \citep{nemirovskij1983problem,agarwal2009information,raginsky2011information}. The logarithmic lower bound of \citet{harvey2019tight} concerns a \emph{nonsmooth last iterate}; it does not give a smooth-convex lower bound for uniform averaging. Lower bounds for nonconvex stationarity \citep{arjevani2023lower,jiang2025provable} and structured nonconvex objectives \citep{saad2025new} concern different settings. Recent smooth-convex SGD analyses include last-iterate upper bounds \citep{liu2023revisiting,garrigos2025last}. We are not aware of an earlier $\Omega(\log T/\sqrt T)$ lower bound for the exact combination considered here: \emph{uniformly averaged unprojected SGD, a smooth convex objective, bounded noise, and a fixed inverse-square-root stepsize schedule}. Theorem~\ref{thm:log-lb} and its SGD extension in Appendix~\ref{app:log-lb} establish this statement directly. This scoped comparison is not a lower bound for all SGD schedules: properly tuned finite-horizon steps and averaging admit different guarantees \citep{nemirovski2009robust}.

\paragraph{Adaptive methods and noise adaptation.}
The online analysis of \citet{mcmahan2010adaptive,duchi2011adaptive} provides the starting point for adaptive normalization. Related developments include adaptive steplength rules and variants of AdaGrad, RMSProp, and Adam \citep{yousefian2012stochastic,mukkamala2017variants,reddi2019convergence}, delayed AdaGrad-type schemes \citep{li2019convergence}, and standard updates under stronger tail assumptions \citep{attia2023sgd,liu2023convergence}. On a bounded convex set, \citet{levy2018online} obtain a log-free $O(LD^2/T+\sigma D/\sqrt T)$ expected bound for scalar AdaGrad. Under bounded iterates and component-function assumptions, \citet{vaswani2020adaptive} derive an analogous noise-adaptive transition using regret, self-bounding, and a quadratic inequality. Their residual-noise parameter differs from the conditional variance bound used here. Coordinate-wise adaptation under anisotropic smoothness and noise is studied by \citet{liu2025adagrad}. Our Theorem~\ref{thm:norm upper bound} keeps constants explicit under the global conditional ABC model. A broader nonconvex literature studies scalar and coordinate-wise normalization \citep{defossez2020simple,kavis2022high,faw2022power,faw2023beyond,liu2023high,yang2023two,wang2023convergence,hong2024revisiting}.
\section{Preliminaries and setup}
\paragraph{Problem and notation.}
We consider
\begin{equation}
 \min_{x\in\mathbb R^d} f(x), \tag{Opt}
\end{equation}
where $f$ is differentiable, convex, and attains its minimum. Fix $x^\star\in\arg\min f$ and write
\[
 f^\star=f(x^\star),\qquad \Delta_t=f(x_t)-f^\star,
 \qquad \bar x_T=\frac1T\sum_{t=1}^T x_t.
\]
When a hard objective depends on the horizon, we write $f_T$ and $f_T^\star$ explicitly. The corresponding trajectory is always generated on that chosen instance. We use Euclidean norms, $u_+=\max\{u,0\}$, and the usual floor and ceiling functions. For nonnegative sequences, $a_T=O(b_T)$ and $a_T=\Omega(b_T)$ denote eventual upper and lower bounds up to positive constants; $a_T\asymp b_T$ means both. Unless stated otherwise, constants may depend on fixed problem and algorithm parameters, but not on $T$ or the failure probability.

\paragraph{Oracle and filtration.}
At $x_t$, the oracle returns $g_t=\nabla f(x_t)+\xi_t$. The initialization $x_1$ is deterministic, $\mathcal F_0$ is trivial, and $\mathcal F_t=\sigma(g_1,\ldots,g_t)$, so $x_t$ is $\mathcal F_{t-1}$-measurable. Every conditional assumption below refers to the history before the current query and holds for all $t\ge1$. The oracle may depend on time and history; independent identically distributed noise is imposed only where stated. We distinguish one fixed infinite-horizon instance from a family whose objective or oracle changes with $T$.

To conduct the theoretical analysis, we introduce some assumptions below.
\begin{assumption}\label{ass:smoothness}
For some $L>0$, for all $x,y\in\mathbb R^d$, $$\|\nabla f(x)-\nabla f(y)\|\le L\|x-y\|.$$
\end{assumption}

\begin{assumption}\label{ass:unbiased}
The oracle satisfies $\mathbb E[g_t\mid\mathcal F_{t-1}]=\nabla f(x_t)$.
\end{assumption}
These two assumptions are used throughout. The results differ in which of the following moment conditions is imposed.
\begin{assumption}[Sub-Gaussian-square noise]\label{ass:sub-Gaussian noise}
For some $\sigma>0$,
\[
 \mathbb E\!\left[\exp\!\left(\frac{\|\xi_t\|^2}{\sigma^2}\right)\middle|\mathcal F_{t-1}\right]\le e.
\]
\end{assumption}
\begin{assumption}[Bounded variance]\label{ass:bounded variance noise}
For some $\sigma\ge0$, $\mathbb E[\|\xi_t\|^2\mid\mathcal F_{t-1}]\le\sigma^2$.
\end{assumption}
\begin{assumption}[ABC condition]\label{ass:ABC condition}
There are $A,B,C\ge0$ such that
\begin{equation}\label{eq:ABC}
 \mathbb E[\|g_t\|^2\mid\mathcal F_{t-1}]
 \le A\Delta_t+B\|\nabla f(x_t)\|^2+C.
\end{equation}
\end{assumption}
\begin{remark}\label{remark}
Conditional Jensen's inequality shows that Assumption~\ref{ass:sub-Gaussian noise} implies Assumption~\ref{ass:bounded variance noise}. Together with unbiasedness, bounded variance implies the ABC condition with $(A,B,C)=(0,1,\sigma^2)$. The ABC framework therefore includes both preceding noise models \citep{khaled2020better}. The case $C=0$ permits state-dependent noise; it does not require every gradient to be deterministic.
\end{remark}

\paragraph{Algorithms.}
For $b_0,\eta>0$, AdaGrad-Norm uses
\begin{equation}\label{eq:adagrad-norm}
 b_t^2=b_0^2+\sum_{s=1}^t\|g_s\|^2,
 \qquad x_{t+1}=x_t-\eta\frac{g_t}{b_t}. \tag{AdaGrad-Norm}
\end{equation}
For $b_{0,i},\eta>0$, coordinate-wise AdaGrad uses
\begin{equation}\label{eq:adagrad-coordinate}
 b_{t,i}^2=b_{0,i}^2+\sum_{s=1}^t g_{s,i}^2,
 \qquad x_{t+1,i}=x_{t,i}-\eta\frac{g_{t,i}}{b_{t,i}},
 \quad 1\le i\le d. \tag{AdaGrad}
\end{equation}
Both updates are unprojected, and each denominator includes the current gradient; it is generally not $\mathcal F_{t-1}$-measurable. In dimension one, the algorithms coincide when their initial accumulators agree. All lower bounds exploit this common scalar update, whereas the upper bound treats the two normalizations separately and allows arbitrary positive $b_{0,i}$.

\section{Main results}
We first rule out a universal last-iterate rate, then show that averaging does not remove the high-probability obstruction under finite variance. The next two results address expected average error without a common stability radius. The final theorem gives the complementary noise-adaptive upper bound when stability is available.

\subsection{No universal vanishing last-iterate rate}
Our first result concerns the expected error at the actual final iterate. Neither smoothness, sub-Gaussian noise, nor an arbitrarily small trajectory radius supplies a universal vanishing rate in the conditional-oracle model.

\begin{theorem}[Arbitrarily slow last-iterate convergence]\label{thm:no last-iterate rate}
Fix $L,\sigma,b_0,\eta,D>0$. For every prescribed sequence $r_T>0$ with $r_T\to0$, there exist a single convex $L$-smooth quadratic $f:\mathbb R\to\mathbb R$, a single infinite-horizon oracle satisfying Assumptions~\ref{ass:unbiased} and~\ref{ass:sub-Gaussian noise}, and deterministic $T_n\to\infty$, such that both algorithms, initialized at $x_1=x^\star=0$, satisfy $|x_t-x^\star|\le D$ almost surely for all $t\ge1$ and
\begin{equation}\label{eq:arbitrary-slow}
 \frac{\mathbb E[f(x_{T_n})-f^\star]}{r_{T_n}}\longrightarrow\infty.
\end{equation}
\end{theorem}
\begin{proof}[Proof sketch]
At the minimizer of a small quadratic, both the iterate and accumulator stay unchanged until a nonzero pulse occurs. At scheduled times $\tau_n$, allow symmetric noise of fixed magnitude with probability $q/(n+1)^2$, and switch off all noise after the first pulse. That pulse creates a fixed displacement, followed by deterministic contraction. Its occurrence probability gives $\mathbb E[\Delta_{\tau_n+1}]\ge c_1/(n+1)^2$; see \eqref{eq:spike-lower}. Choosing $\tau_n$ so that $r_{\tau_n+1}\le(n+1)^{-4}$ proves the claim. The delay, rather than growing noise magnitude, defeats the rate. Appendix~\ref{app:last} also verifies asymptotic convergence of this particular instance.
\end{proof}

\begin{remark}\label{rem:last-quantifiers}
The quantifiers select an instance after the proposed sequence $r_T$, but then hold it fixed for all evaluation times. Thus the theorem is stronger than excluding every polynomial exponent separately, and is not a horizon-by-horizon change of oracle. The quantifier is $\forall(r_T\to0)\,\exists\mathcal I\,\exists(T_n)$, not the existence of one instance slower than every possible sequence. The constructed instance satisfies $\mathbb E[\Delta_t]\to0$; no general qualitative convergence theorem is asserted. An explicit schedule, independent of a proposed rate, gives
\begin{equation}\label{eq:double-log-last}
 \mathbb E[\Delta_{T_n}]\ge\frac{c}{(1+\log\log T_n)^2},
 \qquad T_n=\lceil\exp(\exp n)\rceil+1,
\end{equation}
where $c>0$ depends only on $L,\sigma,b_0,\eta,D$. The schedule for an arbitrary $r_T$ need not be this particular one. The obstruction uses delayed, time- and history-dependent noise, not unbounded pulse magnitudes; it does not rule out rates under additional oracle restrictions.
\end{remark}

\begin{remark}\label{rem:last-fixed-data}
The construction uses a fixed objective and bounded pulses. The objective is a strongly convex quadratic with curvature $\ell=\min\{L,b_0/(2\eta)\}$, independent of the proposed rate. Its minimizer, the initialization, the radius, and the pulse magnitude $M$ are also fixed independently of $r_T$. In fact, the construction satisfies the stronger pathwise bounds $|\xi_t|\le M$ and $|g_t|\le \ell d_0+M$, with $d_0\le D$. Only the scheduled pulse times depend on $r_T$. Thus the obstruction is neither a flattening sequence of objectives nor increasingly large gradients. It arises because elapsed time need not force growth of the adaptive accumulator: the oracle can remain noiseless at a minimizer for an arbitrarily long interval. The uniform moment assumptions control pulse size and probability but do not impose a clock on when stochastic variation must occur.
\end{remark}

\begin{remark}\label{rem:last-versus-average}
Averaging can still have a rate because it dilutes a late excursion: before the pulse, the objective error is zero and the accumulator stays at $b_0$; immediately afterward, the last iterate has a fixed error. A late evaluation can therefore detect the entire displacement, whereas uniform averaging dilutes a short late excursion by the horizon. This explains why Theorem~\ref{thm:norm upper bound} can give an explicit expected average-iterate rate on the same class. An average-iterate upper bound cannot be converted into a last-iterate rate merely by substituting the final point.
\end{remark}

\subsection{A high-probability limitation even for the average iterate}
We next weaken the noise assumption to finite variance and ask for a high-probability guarantee for the \emph{uniform average}, rather than the last iterate. The following result combines one fixed infinite-horizon instance, this averaged output, and any prescribed bounded-trajectory radius. 

\begin{theorem}[No polynomial average-iterate rate with polylogarithmic confidence]\label{thm:hp}
Fix $L,\sigma,b_0,\eta,D>0$. There exist a fixed one-dimensional convex $L$-smooth quadratic and a fixed infinite-horizon oracle satisfying Assumptions~\ref{ass:unbiased} and~\ref{ass:bounded variance noise}, such that both algorithms, initialized at $x_1=x^\star=0$, satisfy $|x_t-x^\star|\le D$ almost surely for every $t\ge1$. There also exist deterministic $T_n\to\infty$, $\delta_n\in(0,1/2)$, and constants $c_\star,K>0$ such that
\begin{equation}\label{eq:hp-lower}
 \begin{aligned}
 \mathbb P\{f(\bar x_{T_n})-f^\star\ge c_\star\}&\ge2\delta_n,\\
 \log(1/\delta_n)&\le K(1+\log T_n)^2.
 \end{aligned}
\end{equation}
Consequently, no bound
\begin{equation}\label{eq:hp-excluded}
 f(\bar x_T)-f^\star\le K_0\bigl[T^{-a}+P(\log(1/\delta))T^{-b}\bigr],
 \qquad a,b>0,
\end{equation}
can hold with probability at least $1-\delta$ for every $T\ge1$ and $\delta\in(0,1/2)$, with $K_0<\infty$ and a fixed polynomial $P$, nonnegative on $[\log 2,\infty)$, independent of $T,\delta$.
\end{theorem}
\begin{proof}[Proof sketch]
A single large gradient can move the iterate almost $\eta$, so it cannot respect an arbitrary radius $D<\eta$. Instead, arrange a finite random cascade whose designated path alternates between $0$ and $-d_0$, while the accumulator grows geometrically; see \eqref{eq:hp-designated}. Rare pulses preserve a common conditional variance bound, and every departure switches off the noise. Stopping after an odd number $n$ of pulses leaves a nonzero displacement and an exponentially large accumulator, so the displacement survives for an exponentially long horizon. The designated event has probability at least $\exp(-O(n^2))$, yielding \eqref{eq:hp-lower}. Appendix~\ref{app:hp} checks all paths and the fixed-oracle quantifier.
\end{proof}

\begin{remark}\label{rem:hp-average}
Our result concerns the averaged output on one fixed bounded instance. The last-iterate event in \citet[Theorem~B.4]{chezhegov2025clipping} does not by itself lower-bound $f(\bar x_T)-f^\star$. Convexity gives an \emph{upper} bound on $f(\bar x_T)$ by the mean of the iterate values, not the lower bound needed here. Our cascade instead creates a displacement that persists for a substantial fraction of the horizon. On the designated event all iterates are nonpositive and the late ones have magnitude at least $d_0/2$, so their uniform average has magnitude at least $d_0/4$; see \eqref{eq:hp-average}. The argument directly excludes the stated confidence dependence even for this more forgiving output.

The fixed-instance quantifier is particularly strong here. Once the cascade law has been chosen from $L,\sigma,b_0,\eta,D$, the \emph{same} instance excludes every pair of positive exponents and every fixed polynomial appearing in \eqref{eq:hp-excluded}. It is not retuned for a proposed confidence bound. Small oscillations consume normalization while leaving little net displacement; after the last designated pulse, the accumulated normalization prevents rapid recovery. This separates trajectory containment from concentration of the averaged objective error. In contrast, a single large initial pulse need not respect a prescribed radius smaller than $\eta$.
\end{remark}

\begin{remark}\label{rem:hp-confidence}
The obstruction concerns confidence dependence rather than nonconvergence: the levels $\delta_n$ shrink with the horizons. The theorem does not say that the averaged error fails to vanish for a fixed $\delta$, nor that averaging fails almost surely on the hard instance. Indeed, Theorem~\ref{thm:norm upper bound} and Markov's inequality give, in dimension one, for every $\delta>0$,
\begin{equation}\label{eq:markov-average}
 \mathbb P\!\left\{f(\bar x_T)-f^\star>
 \frac1\delta\left(\frac{2L\Gamma^2+\Gamma b_0}{T}
                 +\frac{\Gamma\sigma}{\sqrt T}\right)\right\}\le\delta,
 \quad \Gamma=\frac{D^2}{2\eta}+\eta.
\end{equation}
Its $1/\delta$ dependence is outside the class excluded by \eqref{eq:hp-excluded}. Thus the substantive distinction is that \emph{uniform averaging does not upgrade finite variance to a polynomial rate with polylogarithmic confidence dependence}, even on one fixed bounded instance.
\end{remark}

\subsection{Finite variance permits constant worst-case average error}\label{sec:no-bounded}
The previous theorem retains a bounded trajectory and isolates confidence dependence. If that stability condition is removed, finite variance alone does not even supply a vanishing \emph{uniform worst-case} expected average-error bound for fixed algorithm parameters.

\begin{theorem}[Constant worst-case error under bounded variance]\label{thm:no-bounded}
Fix $\sigma,b_0,\eta>0$. For every sufficiently large integer $T$, there exist a one-dimensional convex $1$-smooth function $f_T$ with $f_T^\star=0$ and an oracle with independent identically distributed additive noise satisfying Assumptions~\ref{ass:unbiased} and~\ref{ass:bounded variance noise}, such that both algorithms, initialized at $x_1=1$ with $x^\star=0$, satisfy
\begin{equation}\label{eq:variance-worstcase}
 \mathbb E[f_T(\bar x_T)-f_T^\star]
 \ge\frac{\eta\sigma^2}{256\sqrt{b_0^2+\sigma^2/16}}.
\end{equation}
The objective and noise law may depend on $T$, but the initial distance, algorithm parameters, variance bound, and displayed lower-bound constant do not.
\end{theorem}
\begin{proof}[Proof sketch]
Use a smooth hinge with slope $a=\sigma/(4\sqrt T)$. In its linear region, the oracle returns $-a$ with probability $1-1/T$; a rare large positive correction preserves unbiasedness and bounded variance. With probability at least $1/4$, no correction occurs before $T$. On that event the accumulator stays bounded, the iterates move away from the minimizer, and their average displacement is of order $\eta aT$. The objective gap is therefore of order $\eta a^2T$, a positive constant in $T$. Equations~\eqref{eq:variance-trajectory}--\eqref{eq:variance-event-gap} in Appendix~\ref{app:no-bounded} make the constants explicit.
\end{proof}

\begin{remark}\label{rem:variance-scope}
This horizon-wise lower bound concerns uniform worst-case convergence, not a failure of convergence on one fixed instance. Unlike the preceding pulse constructions, it uses independent identically distributed additive noise on every fixed instance. The initial distance remains one across the family; the obstruction cannot be attributed to a growing initialization error.
\end{remark}

\subsection{A logarithmic average-error lower bound under sub-Gaussian noise}\label{sec:log-lb}
The finite-variance counterexamples allow increasingly large gradients. The next construction shows that a loss persists even with uniformly bounded noise. A single scalar, fixed-magnitude oracle proves the result for both algorithms, without an auxiliary coordinate.

\begin{theorem}[A scalar logarithmic average-error lower bound]\label{thm:log-lb}
Fix $L,\sigma,b_0,\eta>0$. For every sufficiently large integer $T$, there exist a convex $L$-smooth function $f_T:\mathbb R\to\mathbb R$ with $f_T^\star=0$ and an oracle satisfying Assumptions~\ref{ass:unbiased} and~\ref{ass:sub-Gaussian noise}, such that both AdaGrad-Norm and AdaGrad, initialized at $x_1=x^\star=0$, satisfy
\begin{equation}\label{eq:log-lower}
 \mathbb E[f_T(\bar x_T)-f_T^\star]
 \ge c\eta\sigma\frac{\log T}{\sqrt T},
\end{equation}
where $c>0$ is universal. The objective and oracle may depend on $T$, but $L,\sigma,b_0,\eta$ do not. No horizon-independent bounded-iterate radius is assumed.
\end{theorem}
\begin{proof}[Proof sketch]
Choose a smooth hinge with slope $a\asymp\sigma\sqrt{\log T/T}$ and an unbiased oracle with $g_t\in\{-\gamma,\gamma\}$, where $\gamma=\sigma/2$. Then $b_t^2=b_0^2+t\gamma^2$ is deterministic. The average splits into a mean-zero martingale term and a nonnegative drift. The squared martingale weights sum to order $\eta^2\log T$, so second- and fourth-moment estimates give a positive fluctuation of order $\eta\sqrt{\log T}$; see \eqref{eq:log-variance}--\eqref{eq:log-positive}. Choosing the slope constant small makes the drift smaller than this fluctuation. Multiplication by $a$ yields \eqref{eq:log-lower}. Appendix~\ref{app:log-lb} verifies the oracle and smoothing conditions.
\end{proof}

\begin{remark}\label{rem:log-scope}
The scope of the logarithmic lower bound is smooth-convex average function value, rather than the nonconvex stationarity metric of \citet{jiang2025provable}. It does not exclude horizon-dependent parameter tuning, alternative outputs, or modifications such as projection. In particular, it does not establish that every logarithm in every AdaGrad upper bound is necessary.
\end{remark}

\begin{remark}\label{rem:sgd-consequence}
The same construction has a consequence for inverse-square-root SGD: since $g_t^2=\gamma^2$ on this oracle, AdaGrad is exactly unprojected SGD with deterministic stepsizes $\eta/\sqrt{b_0^2+t\gamma^2}$. More generally, fix $L,\sigma,\lambda,\tau>0$ and consider
\begin{equation}\label{eq:sgd-schedule}
 x_{t+1}=x_t-\frac{\lambda}{\sqrt{t+\tau}}g_t.
\end{equation}
For every sufficiently large $T$, there exists a convex $L$-smooth scalar objective $f_T$, with $x_1=x^\star=0$, for which even the independent additive-noise oracle $g_t=f_T'(x_t)+\xi_t$, $\mathbb P(\xi_t=\sigma/2)=\mathbb P(\xi_t=-\sigma/2)=1/2$, satisfies
\begin{equation}\label{eq:sgd-lower}
 \mathbb E[f_T(\bar x_T)-f_T^\star]
 \ge c_{\mathrm{SGD}}\lambda\sigma^2\frac{\log T}{\sqrt T},
\end{equation}
where $c_{\mathrm{SGD}}>0$ is universal. The final part of Appendix~\ref{app:log-lb} gives the full reduction and checks the independent-noise version. This is a lower bound for a smooth-convex \emph{average iterate}, not the nonsmooth last-iterate lower bound of \citet{harvey2019tight}. We are not aware of an earlier result for this exact combination of assumptions and schedule. The qualification on the schedule is essential: the claim does not extend to all SGD stepsizes, suffix averages, or projected variants.
\end{remark}

\begin{remark}\label{rem:sgd-schedule}
The stepsize schedule matters and cannot be omitted from the statement. For example, for SGD with a constant step $0<\alpha\le1/(2L)$ on a fixed horizon and a conditionally unbiased oracle of variance at most $\sigma^2$, the squared-distance recursion and $\|\nabla f(x)\|^2\le2L(f(x)-f^\star)$ give
\begin{equation}\label{eq:sgd-fixed-horizon}
 \mathbb E[f(\bar x_T)-f^\star]
 \le\frac{\|x_1-x^\star\|^2}{\alpha T}+\alpha\sigma^2.
\end{equation}
Indeed, one update decreases expected squared distance by at least $\alpha\mathbb E[\Delta_t]-\alpha^2\sigma^2$; summation and convexity prove the display. A horizon-dependent choice $\alpha\asymp T^{-1/2}$ therefore gives $O(T^{-1/2})$. On the lower-bound instance $x_1=x^\star$, even the initial-distance term vanishes. The logarithm in \eqref{eq:sgd-lower} is a limitation of the fixed inverse-square-root schedule with uniform averaging, not of the stochastic oracle problem itself. The final part of Appendix~\ref{app:log-lb} also verifies the integrability and distance calculation behind \eqref{eq:sgd-fixed-horizon}.
\end{remark}

\subsection{Noise-adaptive upper bounds under the ABC condition}
The preceding results distinguish instability, heavy-tail confidence dependence, and last-iterate behavior. None prevents a log-free expected average-error bound under bounded iterates. The final theorem makes this complementary guarantee explicit using the classical regret and self-bounding argument \citep{duchi2011adaptive,levy2018online,vaswani2020adaptive,liu2025adagrad}.

\begin{theorem}[Noise-adaptive average-iterate upper bounds]\label{thm:norm upper bound}
Under Assumptions~\ref{ass:smoothness}, \ref{ass:unbiased}, and~\ref{ass:ABC condition}, let $Q=A+2LB$. For every integer $T\ge1$:
\begin{enumerate}[leftmargin=1.5em,itemsep=2pt,topsep=2pt]
\item \textbf{AdaGrad-Norm.} If $\|x_t-x^\star\|\le D$ almost surely for all $t$, set $\Gamma=D^2/(2\eta)+\eta$. Then
\begin{equation}\label{eq:norm-noise-adaptive}
 \mathbb E[f(\bar x_T)-f^\star]
 \le\frac{\Gamma^2Q+\Gamma b_0}{T}+\Gamma\sqrt{\frac CT}.
\end{equation}
\item \textbf{AdaGrad.} If $|x_{t,i}-x_i^\star|\le D_i$ almost surely for all $t,i$, set $\Gamma_i=D_i^2/(2\eta)+\eta$, $\Lambda=(\sum_i\Gamma_i^2)^{1/2}$, and $\beta_0=(\sum_i b_{0,i}^2)^{1/2}$. Then
\begin{equation}\label{eq:coordinate-noise-adaptive}
 \mathbb E[f(\bar x_T)-f^\star]
 \le\frac{\Lambda^2Q+\Lambda\beta_0}{T}+\Lambda\sqrt{\frac CT}.
\end{equation}
\end{enumerate}
The radii are nonnegative, finite, and independent of $t$. In particular, $C=0$ gives $O(T^{-1})$; bounded variance permits $Q=2L$, $C=\sigma^2$ and gives $O(T^{-1}+\sigma T^{-1/2})$.
\end{theorem}
\begin{proof}[Proof sketch]
The key is to retain the objective gap in the second-moment bound instead of replacing it by a uniform constant. For $S_T=\sum_{t=1}^T\mathbb E[\Delta_t]$, smoothness and the ABC condition give $\sum_{t=1}^T\mathbb E[\|g_t\|^2]\le QS_T+CT$. A pathwise regret inequality removes the correlated denominator before taking expectations, yielding
$S_T\le\Gamma\sqrt{b_0^2+QS_T+CT}$ for AdaGrad-Norm. For AdaGrad, coordinate-wise regret and Cauchy--Schwarz give the same inequality with $(\Gamma,b_0)$ replaced by $(\Lambda,\beta_0)$; see \eqref{eq:norm-self-consistency} and \eqref{eq:coordinate-self-consistency}. Solving these quadratic inequalities produces the noise-adaptive bounds. Appendix~\ref{app:positive} supplies separate proofs for the two algorithms.
\end{proof}

\begin{remark}\label{rem:abc-scope}
Noise adaptation here is conditional on stability, although the algorithms need not know $A,B,C,L$ or the stability radii. For a common coordinate initialization, $\beta_0=\sqrt d\,b_0$. The case $C=0$ permits state-dependent stochastic noise; it does not require a deterministic oracle. The root-form refinements and the two separate algorithm proofs are in Appendix~\ref{app:positive}. The theorem assumes stability and therefore does not contradict the lower bounds in Sections~\ref{sec:no-bounded} and~\ref{sec:log-lb}, where a common trajectory radius is absent.
\end{remark}

\section{Conclusion}
Our first result rules out any prescribed uniform vanishing expected last-iterate rate for AdaGrad-Norm and AdaGrad under conditional sub-Gaussian oracles, even within an arbitrarily small trajectory radius. One fixed instance violates the proposed rate along a subsequence while still converging; only the pulse schedule depends on that rate.

Under finite variance, a fixed infinite-horizon instance also defeats polynomial average-iterate rates with polylogarithmic confidence dependence within any prescribed radius. A persistent one-sided displacement survives averaging, but the conclusion remains compatible with fixed-confidence convergence and Markov bounds.

Without a common radius, horizon-dependent families give constant worst-case expected average error under finite variance or a logarithmic lower bound under bounded noise. The latter also applies to fixed inverse-square-root SGD, not every stepsize or output. Under bounded iterates, classical regret analysis gives noise-adaptive ABC upper bounds. Identifying sampling assumptions that restore uniform last-iterate rates for the unchanged updates remains an important direction.

\clearpage
\begingroup
\interlinepenalty=10000
\renewcommand{\bibfont}{\small}
\setlength{\bibsep}{4pt plus 1pt}
\renewcommand{\bibsection}{\section*{\refname}\addcontentsline{toc}{section}{\refname}}
\bibliographystyle{plainnat}
\bibliography{all_noise_references}

@article{duchi2011adaptive,
  title={Adaptive subgradient methods for online learning and stochastic optimization},
  author={Duchi, John and Hazan, Elad and Singer, Yoram},
  journal={Journal of Machine Learning Research},
  volume={12},
  number={61},
  pages={2121--2159},
  year={2011}
}

@inproceedings{mcmahan2010adaptive,
  title={Adaptive bound optimization for online convex optimization},
  author={McMahan, H. Brendan and Streeter, Matthew J.},
  booktitle={COLT},
  pages={244--256},
  year={2010}
}

@article{ward2020adagrad,
  title={{AdaGrad} stepsizes: Sharp convergence over nonconvex landscapes},
  author={Ward, Rachel and Wu, Xiaoxia and Bottou, Leon},
  journal={Journal of Machine Learning Research},
  volume={21},
  number={219},
  pages={1--30},
  year={2020}
}

@inproceedings{li2019convergence,
  title={On the convergence of stochastic gradient descent with adaptive stepsizes},
  author={Li, Xiaoyu and Orabona, Francesco},
  booktitle={The 22nd international conference on artificial intelligence and statistics},
  pages={983--992},
  year={2019},
  organization={PMLR}
}

@inproceedings{liu2023convergence,
  title={On the Convergence of {AdaGrad} (Norm) on {$\mathbb{R}^d$}: Beyond Convexity, Non-Asymptotic Rate and Acceleration},
  author={Liu, Zijian and Nguyen, Ta Duy and Ene, Alina and Nguyen, Huy},
  booktitle={International Conference on Learning Representations},
  year={2023},
  organization={International Conference on Learning Representations}
}

@inproceedings{attia2023sgd,
  title={{SGD} with {AdaGrad} stepsizes: Full adaptivity with high probability to unknown parameters, unbounded gradients and affine variance},
  author={Attia, Amit and Koren, Tomer},
  booktitle={International Conference on Machine Learning},
  pages={1147--1171},
  year={2023},
  organization={PMLR}
}

@inproceedings{hong2024revisiting,
  title={Revisiting convergence of {AdaGrad} with relaxed assumptions},
  author={Hong, Yusu and Lin, Junhong},
  booktitle={Proceedings of the Fortieth Conference on Uncertainty in Artificial Intelligence},
  pages={1727--1750},
  year={2024}
}

@inproceedings{wang2023convergence,
  title={Convergence of {AdaGrad} for non-convex objectives: Simple proofs and relaxed assumptions},
  author={Wang, Bohan and Zhang, Huishuai and Ma, Zhiming and Chen, Wei},
  booktitle={The Thirty Sixth Annual Conference on Learning Theory},
  pages={161--190},
  year={2023},
  organization={PMLR}
}

@inproceedings{kavis2022high,
  title={High Probability Bounds for a Class of Nonconvex Algorithms with {AdaGrad} Stepsize},
  author={Kavis, Ali and Levy, Kfir and Cevher, Volkan},
  booktitle={10th International Conference on Learning Representations (ICLR)},
  year={2022}
}

@inproceedings{faw2022power,
  title={The power of adaptivity in {SGD}: Self-tuning step sizes with unbounded gradients and affine variance},
  author={Faw, Matthew and Tziotis, Isidoros and Caramanis, Constantine and Mokhtari, Aryan and Shakkottai, Sanjay and Ward, Rachel},
  booktitle={Conference on Learning Theory},
  pages={313--355},
  year={2022},
  organization={PMLR}
}

@inproceedings{faw2023beyond,
  title={Beyond uniform smoothness: A stopped analysis of adaptive {SGD}},
  author={Faw, Matthew and Rout, Litu and Caramanis, Constantine and Shakkottai, Sanjay},
  booktitle={The Thirty Sixth Annual Conference on Learning Theory},
  pages={89--160},
  year={2023},
  organization={PMLR}
}

@inproceedings{liu2023high,
  title={High probability convergence of stochastic gradient methods},
  author={Liu, Zijian and Nguyen, Ta Duy and Nguyen, Thien Hang and Ene, Alina and Nguyen, Huy},
  booktitle={International conference on machine learning},
  pages={21884--21914},
  year={2023},
  organization={PMLR}
}

@inproceedings{liu2025adagrad,
  title={{AdaGrad} under Anisotropic Smoothness},
  author={Liu, Yuxing and Pan, Rui and Zhang, Tong},
  booktitle={13th International Conference on Learning Representations, ICLR 2025},
  year={2025},
  organization={International Conference on Learning Representations, ICLR}
}

@inproceedings{jiang2025provable,
  title={Provable Complexity Improvement of {AdaGrad} over {SGD}: Upper and Lower Bounds in Stochastic Non-Convex Optimization},
  author={Jiang, Ruichen and Maladkar, Devyani and Mokhtari, Aryan},
  booktitle={The Thirty Eighth Annual Conference on Learning Theory},
  pages={3124--3158},
  year={2025},
  organization={PMLR}
}

@inproceedings{harvey2019tight,
  title={Tight analyses for non-smooth stochastic gradient descent},
  author={Harvey, Nicholas JA and Liaw, Christopher and Plan, Yaniv and Randhawa, Sikander},
  booktitle={Conference on Learning Theory},
  pages={1579--1613},
  year={2019},
  organization={PMLR}
}

@book{nemirovskij1983problem,
  title={Problem complexity and method efficiency in optimization},
  author={Nemirovskij, Arkadij Semenovi{\v{c}} and Yudin, David Borisovich},
  year={1983},
  publisher={Wiley-Interscience}
}

@article{agarwal2009information,
  title={Information-theoretic lower bounds on the oracle complexity of convex optimization},
  author={Agarwal, Alekh and Wainwright, Martin J and Bartlett, Peter and Ravikumar, Pradeep},
  journal={Advances in Neural Information Processing Systems},
  volume={22},
  year={2009}
}

@article{raginsky2011information,
  title={Information-based complexity, feedback and dynamics in convex programming},
  author={Raginsky, Maxim and Rakhlin, Alexander},
  journal={IEEE Transactions on Information Theory},
  volume={57},
  number={10},
  pages={7036--7056},
  year={2011},
  publisher={IEEE}
}

@article{arjevani2023lower,
  title={Lower bounds for non-convex stochastic optimization},
  author={Arjevani, Yossi and Carmon, Yair and Duchi, John C and Foster, Dylan J and Srebro, Nathan and Woodworth, Blake},
  journal={Mathematical Programming},
  volume={199},
  number={1--2},
  pages={165--214},
  year={2023},
  publisher={Springer}
}

@article{nemirovski2009robust,
  title={Robust stochastic approximation approach to stochastic programming},
  author={Nemirovski, Arkadi and Juditsky, Anatoli and Lan, Guanghui and Shapiro, Alexander},
  journal={SIAM Journal on optimization},
  volume={19},
  number={4},
  pages={1574--1609},
  year={2009},
  publisher={SIAM}
}

@article{saad2025new,
  title={New lower bounds for stochastic non-convex optimization through divergence decomposition},
  author={Saad, El Mehdi and Lee, Wei-Cheng and Orabona, Francesco},
  journal={arXiv preprint arXiv:2502.14060},
  year={2025}
}

@inproceedings{shamir2013stochastic,
  title={Stochastic gradient descent for non-smooth optimization: Convergence results and optimal averaging schemes},
  author={Shamir, Ohad and Zhang, Tong},
  booktitle={International conference on machine learning},
  pages={71--79},
  year={2013},
  organization={PMLR}
}

@inproceedings{zhang2004solving,
  title={Solving large scale linear prediction problems using stochastic gradient descent algorithms},
  author={Zhang, Tong},
  booktitle={Proceedings of the twenty-first international conference on Machine learning},
  pages={116},
  year={2004}
}

@article{bertsekas2000gradient,
  title={Gradient convergence in gradient methods with errors},
  author={Bertsekas, Dimitri P and Tsitsiklis, John N},
  journal={SIAM Journal on Optimization},
  volume={10},
  number={3},
  pages={627--642},
  year={2000},
  publisher={SIAM}
}

@inproceedings{liu2022almost,
  title={On almost sure convergence rates of stochastic gradient methods},
  author={Liu, Jun and Yuan, Ye},
  booktitle={Conference on Learning Theory},
  pages={2963--2983},
  year={2022},
  organization={PMLR}
}

@misc{orabona2020almost,
  author = {Orabona, Francesco},
  title = {Almost Sure Convergence of {SGD} on Smooth Non-Convex Functions},
  howpublished = {Blog post},
  year = {2020},
  url = {https://parameterfree.com/2020/10/05/almost-sure-convergence-of-sgd-on-smooth-non-convex-functions/}
}

@misc{orabona2020unbounded,
  author = {Orabona, Francesco},
  title = {Last Iterate of {SGD} Converges (Even in Unbounded Domains)},
  howpublished = {Blog post},
  year = {2020},
  url = {https://parameterfree.com/2020/08/07/last-iterate-of-sgd-converges-even-in-unbounded-domains/}
}

@article{liu2023revisiting,
  title={Revisiting the last-iterate convergence of stochastic gradient methods},
  author={Liu, Zijian and Zhou, Zhengyuan},
  journal={arXiv preprint arXiv:2312.08531},
  year={2023}
}

@inproceedings{li2022last,
  title={On the last iterate convergence of momentum methods},
  author={Li, Xiaoyu and Liu, Mingrui and Orabona, Francesco},
  booktitle={International Conference on Algorithmic Learning Theory},
  pages={699--717},
  year={2022},
  organization={PMLR}
}

@article{hong2024convergence,
  title={On convergence of {Adam} for stochastic optimization under relaxed assumptions},
  author={Hong, Yusu and Lin, Junhong},
  journal={Advances in Neural Information Processing Systems},
  volume={37},
  pages={10827--10877},
  year={2024}
}

@article{gorbunov2020stochastic,
  title={Stochastic optimization with heavy-tailed noise via accelerated gradient clipping},
  author={Gorbunov, Eduard and Danilova, Marina and Gasnikov, Alexander},
  journal={Advances in Neural Information Processing Systems},
  volume={33},
  pages={15042--15053},
  year={2020}
}

@article{gorbunov2021near,
  title={Near-optimal high probability complexity bounds for non-smooth stochastic optimization with heavy-tailed noise},
  author={Gorbunov, Eduard and Danilova, Marina and Shibaev, Innokentiy and Dvurechensky, Pavel and Gasnikov, Alexander},
  journal={arXiv preprint arXiv:2106.05958v1},
  year={2021},
}

@article{cutkosky2021high,
  title={High-probability bounds for non-convex stochastic optimization with heavy tails},
  author={Cutkosky, Ashok and Mehta, Harsh},
  journal={Advances in Neural Information Processing Systems},
  volume={34},
  pages={4883--4895},
  year={2021}
}

@inproceedings{li2020high,
  title={A High Probability Analysis of Adaptive {SGD} with Momentum},
  author={Li, Xiaoyu and Orabona, Francesco},
  booktitle={Workshop on Beyond First Order Methods in ML Systems at ICML'20},
  year={2020}
}

@inproceedings{liu2024high,
  title={High-probability bound for non-smooth non-convex stochastic optimization with heavy tails},
  author={Liu, Langqi and Wang, Yibo and Zhang, Lijun},
  booktitle={Forty-first International Conference on Machine Learning},
  year={2024}
}

@inproceedings{chezhegov2025clipping,
  title={Clipping Improves {Adam}-Norm and {AdaGrad}-Norm when the Noise Is Heavy-Tailed},
  author={Chezhegov, Savelii and Klyukin, Yaroslav and Semenov, Andrei and Beznosikov, Aleksandr and Gasnikov, Alexander and Horv{\'a}th, Samuel and Tak{\'a}{\v{c}}, Martin and Gorbunov, Eduard},
  booktitle={International Conference on Machine Learning},
  pages={10269--10333},
  year={2025},
  organization={PMLR}
}

@article{defossez2020simple,
  title={A simple convergence proof of {Adam} and {AdaGrad}},
  author={D{\'e}fossez, Alexandre and Bottou, L{\'e}on and Bach, Francis and Usunier, Nicolas},
  journal={arXiv preprint arXiv:2003.02395},
  year={2020}
}

@article{yang2023two,
  title={Two sides of one coin: the limits of untuned {SGD} and the power of adaptive methods},
  author={Yang, Junchi and Li, Xiang and Fatkhullin, Ilyas and He, Niao},
  journal={Advances in Neural Information Processing Systems},
  volume={36},
  pages={74257--74288},
  year={2023}
}

@inproceedings{mukkamala2017variants,
  title={Variants of {RMSProp} and {AdaGrad} with logarithmic regret bounds},
  author={Mukkamala, Mahesh Chandra and Hein, Matthias},
  booktitle={International conference on machine learning},
  pages={2545--2553},
  year={2017},
  organization={PMLR}
}

@article{reddi2019convergence,
  title={On the convergence of {Adam} and beyond},
  author={Reddi, Sashank J and Kale, Satyen and Kumar, Sanjiv},
  journal={arXiv preprint arXiv:1904.09237},
  year={2019}
}

@article{yousefian2012stochastic,
  title={On stochastic gradient and subgradient methods with adaptive steplength sequences},
  author={Yousefian, Farzad and Nedi{\'c}, Angelia and Shanbhag, Uday V},
  journal={Automatica},
  volume={48},
  number={1},
  pages={56--67},
  year={2012},
  publisher={Elsevier}
}

@article{khaled2020better,
  title={Better theory for {SGD} in the nonconvex world},
  author={Khaled, Ahmed and Richt{\'a}rik, Peter},
  journal={arXiv preprint arXiv:2002.03329},
  year={2020}
}

@article{preobrazhenskaia2026last,
  title={Last Iterate Convergence of {AdaGrad}-Norm for Convex Non-Smooth Optimization},
  author={Preobrazhenskaia, Margarita and Sidorov, Makar and Preobrazhenskii, Igor and Gorbunov, Eduard},
  journal={arXiv preprint arXiv:2604.10728},
  year={2026}
}

@inproceedings{levy2018online,
  title={Online Adaptive Methods, Universality and Acceleration},
  author={Levy, Kfir Y. and Yurtsever, Alp and Cevher, Volkan},
  booktitle={Advances in Neural Information Processing Systems},
  volume={31},
  year={2018}
}

@article{vaswani2020adaptive,
  title={Adaptive Gradient Methods Converge Faster with Over-Parameterization (but you should do a line-search)},
  author={Vaswani, Sharan and Laradji, Issam and Kunstner, Frederik and Meng, Si Yi and Schmidt, Mark and Lacoste-Julien, Simon},
  journal={arXiv preprint arXiv:2006.06835},
  year={2020}
}

@article{jin2022convergence,
  title={On the Convergence of {mSGD} and {AdaGrad} for Stochastic Optimization},
  author={Jin, Ruinan and Xing, Yu and He, Xingkang},
  journal={arXiv preprint arXiv:2201.11204},
  year={2022}
}

@article{jin2026longrun,
  title={Long-Run Dynamics of {AdaGrad-Norm}: Trajectory Stability and Asymptotic Stationarity},
  author={Jin, Ruinan and Wang, Xiaoyu},
  journal={arXiv preprint arXiv:2601.01853v2},
  year={2026},
  note={Version 2, August 30, 2026}
}

@article{garrigos2025last,
  title={Last-Iterate Complexity of {SGD} for Convex and Smooth Stochastic Problems},
  author={Garrigos, Guillaume and Cortild, Daniel and Ketels, Lucas and Peypouquet, Juan},
  journal={arXiv preprint arXiv:2507.14122},
  year={2025}
}
\appendix
\newpage
\section*{APPENDIX}
The proofs follow the original order of the main results. Appendix~\ref{app:last} first proves the absence of a universal last-iterate rate on a fixed instance. Appendix~\ref{app:hp} treats the uniformly averaged output at shrinking confidence levels, still on a fixed infinite-horizon bounded instance. Appendices~\ref{app:no-bounded} and~\ref{app:log-lb} then construct horizon-dependent expected average-error lower bounds, and the final part of Appendix~\ref{app:log-lb} proves the SGD consequence. Appendix~\ref{app:positive} provides separate proofs of the two noise-adaptive upper bounds. Each lower-bound construction for AdaGrad is scalar, so the two algorithms have the same trajectory.

\section{Proof of Theorem~\ref{thm:no last-iterate rate}}\label{app:last}
We begin with the last-iterate limitation, for which the entire oracle must remain fixed while the evaluation times tend to infinity. The construction below separates the timing of stochastic variation from its magnitude, so a bounded pulse can reveal slow convergence at an arbitrarily late time.

\begin{proof}
The construction starts at the minimizer and keeps the accumulator unchanged until the first nonzero pulse. Choose the objective curvature and pulse size so that a pulse respects the stability radius and every later noise-free update contracts. Specifically, set
\begin{equation}\label{eq:spike-constants}
 \ell=\min\{L,b_0/(2\eta)\},\qquad
 d_0=\min\{D/2,\eta/2\},\qquad
 M=\frac{b_0d_0}{\sqrt{\eta^2-d_0^2}},\qquad
 f(x)=\frac\ell2x^2.
\end{equation}
Initialize at $x_1=x^\star=0$, and let
\[
 q=\min\left\{\frac14,\frac{e-1}{\exp(M^2/\sigma^2)-1}\right\},
 \qquad q_n=\frac{q}{(n+1)^2}.
\]
Given a deterministic strictly increasing sequence $(\tau_n)_{n\ge1}$ of positive integers, allow noise at time $t=\tau_n$ only if no earlier nonzero pulse has occurred. At such an active time, draw $\xi_t=M,-M,0$ with probabilities $q_n/2,q_n/2,1-q_n$, using fresh randomness. At all other times, including every time after the first nonzero pulse, set $\xi_t=0$. Define $g_t=\ell x_t+\xi_t$.

The noise is conditionally symmetric and hence unbiased. At active times,
\begin{equation}\label{eq:spike-mgf}
 \mathbb E[e^{\xi_t^2/\sigma^2}\mid\mathcal F_{t-1}]
 =1+q_n(e^{M^2/\sigma^2}-1)\le e;
\end{equation}
at inactive times the conditional expectation is $1$. Thus one fixed infinite-horizon oracle satisfies the required moment bound for any choice of the schedule.

With the oracle specified, we next verify stability and quantify the displacement. Before the first pulse, $x_t=0$ and $b_{t-1}=b_0$, regardless of how much time has elapsed. If the first pulse occurs at $\tau_n$, the next iterate satisfies
\[
 b_{\tau_n}=\sqrt{b_0^2+M^2},\qquad
 |x_{\tau_n+1}|=\frac{\eta M}{\sqrt{b_0^2+M^2}}=d_0.
\]
Afterward the noise is zero and
\begin{equation}\label{eq:spike-contraction}
 x_{t+1}=(1-\eta\ell/b_t)x_t,\qquad 0<\eta\ell/b_t\le1/2.
\end{equation}
Consequently $|x_t|\le d_0\le D$ on every path.

Let $A_n$ be the event that the first nonzero pulse occurs at $\tau_n$. Since $\sum_{j\ge1}q_j<q\le1/4$,
\[
 \mathbb P(A_n)=q_n\prod_{j<n}(1-q_j)\ge\frac34q_n.
\]
With $T_n=\tau_n+1$, the fixed displacement on $A_n$ therefore yields
\begin{equation}\label{eq:spike-lower}
 \mathbb E[\Delta_{T_n}]\ge\frac{3\ell d_0^2q}{8(n+1)^2}.
\end{equation}
For the explicit schedule $\tau_n=\lceil\exp(\exp n)\rceil$, we have $\log\log T_n\ge n$. Hence \eqref{eq:double-log-last} follows with $c=3\ell d_0^2q/8$.

For an arbitrary prescribed sequence $r_T>0$ tending to zero, choose the times recursively so that $r_{\tau_n+1}\le(n+1)^{-4}$. Such increasing times exist because $r_T\to0$. Equation~\eqref{eq:spike-lower} then gives
\[
 \frac{\mathbb E[\Delta_{T_n}]}{r_{T_n}}
 \ge\frac{3\ell d_0^2q}{8}(n+1)^2\longrightarrow\infty.
\]
Once this schedule has been chosen, both the objective and the entire oracle are fixed. Only the evaluation times change along the subsequence.

To complete the argument, we verify that the lower bound concerns speed, not failure of convergence. Indeed, every path has at most one nonzero pulse. After it occurs,
$b_t^2\le b_0^2+M^2+\ell^2d_0^2t$, so $\sum_t\eta\ell/b_t=\infty$. Equation~\eqref{eq:spike-contraction} and $1-u\le e^{-u}$ imply $x_t\to0$ on that path. A path with no pulse remains at zero. Since $0\le\Delta_t\le\ell d_0^2/2$, bounded convergence also gives $\mathbb E[\Delta_t]\to0$. This verifies convergence only for the constructed instance; no general fixed-instance convergence theorem is used or asserted.

\end{proof}

\section{Proof of Theorem~\ref{thm:hp}}\label{app:hp}
The preceding proof exploits an accumulator that remains small; we now study the opposite mechanism for the averaged output. A finite cascade increases normalization without leaving a prescribed radius, and the resulting displacement persists long enough to survive uniform averaging.

\begin{proof}
A single sufficiently large gradient can cause a displacement close to $\eta$, so it need not respect a prescribed radius $D<\eta$. We instead prescribe alternating small displacements and fix their law independently of the eventual evaluation horizons. Let
\begin{equation}\label{eq:hp-constants}
 d_0=\min\{D/4,\eta/4\},\qquad
 \ell=\min\{L,b_0/(2\eta)\},\qquad
 f(x)=\frac\ell2x^2,\qquad x_1=x^\star=0.
\end{equation}
To keep each designated displacement equal to $d_0$, define
\[
 r=\frac{d_0}{\sqrt{\eta^2-d_0^2}},\qquad R=\sqrt{1+r^2}>1,
 \qquad M_j=rb_0R^{j-1},\qquad p_j=\min\{1/4,\sigma^2/(4M_j^2)\}.
\]
If the first $j$ gradients follow the designated signs $g_j=(-1)^{j+1}M_j$, then
\begin{equation}\label{eq:hp-designated}
 b_j=b_0R^j,\qquad \frac{\eta M_j}{b_j}=d_0,\qquad
 x_{j+1}=\begin{cases}-d_0,&j\text{ odd},\\0,&j\text{ even}.
 \end{cases}
\end{equation}
These identities follow by induction from $M_j=rb_{j-1}$ on that path and $\eta r/R=d_0$.

We now embed this path in an unbiased oracle. Draw an internal finite odd integer $N$ with
$\mathbb P(N=2k+1)=2^{-(k+1)}$ for $k\ge0$, independently of the fresh random choices used below. The oracle is active at time $j$ only if $j\le N$ and every preceding gradient followed its designated sign. Write $h_j=\ell x_j$. On an active branch, use
\begin{equation}\label{eq:hp-oracle}
 g_j=\begin{cases}
 +M_j,&\text{with probability }\frac{p_j}{2}(1+h_j/M_j),\\
 -M_j,&\text{with probability }\frac{p_j}{2}(1-h_j/M_j),\\
 h_j,&\text{with probability }1-p_j.
 \end{cases}
\end{equation}
After an outcome different from the designated sign, or after time $N$, the oracle permanently returns $g_j=h_j$. Thus one law specifies the oracle for the entire infinite trajectory; no later choice of $T_n$ changes it.

The designated path explains the mechanism, but the theorem requires all outcomes to satisfy the oracle and stability assumptions. We therefore check the active branch and every possible way of leaving it.

On an active branch, \eqref{eq:hp-designated} gives $|x_j|\le d_0$, and
\[
 |h_j|\le\ell d_0\le\frac{b_0d_0}{2\eta}\le\frac{M_1}{2}\le\frac{M_j}{2}.
\]
The probabilities in \eqref{eq:hp-oracle} are consequently valid. Let $\mathcal G_{j-1}=\sigma(N,g_1,\ldots,g_{j-1})$, which also reveals whether the oracle is active. On an active branch, direct calculation gives
\begin{equation}\label{eq:hp-moments}
 \mathbb E[g_j\mid\mathcal G_{j-1}]=h_j,\qquad
 \mathbb E[(g_j-h_j)^2\mid\mathcal G_{j-1}]
 =p_j(M_j^2-h_j^2)\le\frac{\sigma^2}{4}.
\end{equation}
On an inactive branch, the mean remains $h_j$ and the variance is zero. Conditioning these statements on $\mathcal F_{j-1}\subseteq\mathcal G_{j-1}$ proves Assumptions~\ref{ass:unbiased} and~\ref{ass:bounded variance noise} for the natural filtration. The internal draw of $N$ need not be observable to the algorithm.

For stability, the only undesignated sign sends $0$ to $d_0$ or $-d_0$ to $-2d_0$. A deterministic-gradient outcome instead contracts toward zero, since
$0\le\eta\ell/b_j\le1/2$. After deactivation every update is such a contraction. Thus
\begin{equation}\label{eq:hp-radius}
 |x_j-x^\star|\le2d_0\le D\qquad\text{for every }j\ge1\text{ on every path}.
\end{equation}
No relation such as $D\ge\eta$ is required.

Having verified the assumptions globally, we can focus on rare paths on which the cascade lasts long enough to make later updates very small. For odd $n$, let $E_n$ be the event that $N=n$ and the first $n$ gradients all follow the designated signs. Each such sign has conditional probability at least $p_j/4$. Set
\[
 p_\star=\min\{1/4,\sigma^2/(4r^2b_0^2)\}>0,\qquad
 \pi_n=2^{-(n+1)/2}(p_\star/4)^nR^{-n(n-1)}.
\]
Since $p_j\ge p_\star R^{-2(j-1)}$,
\begin{equation}\label{eq:hp-probability}
 \mathbb P(E_n)\ge2^{-(n+1)/2}\prod_{j=1}^n(p_j/4)\ge\pi_n.
\end{equation}
On $E_n$, $x_{n+1}=-d_0$, $b_n=b_0R^n$, and all subsequent noise vanishes. For sufficiently large odd $n$, choose
\begin{equation}\label{eq:hp-horizon}
 T_n=\left\lfloor\sqrt{b_0R^n}\right\rfloor.
\end{equation}
Then $T_n\ge2(n+1)$ and $\eta\ell T_n/(b_0R^n)\le1/2$. Monotonicity of the accumulator and Bernoulli's inequality imply, for $n+1\le t\le T_n$,
\[
 |x_t|\ge d_0\left(1-\frac{\eta\ell}{b_0R^n}\right)^{t-n-1}
 \ge d_0\left(1-\frac{\eta\ell T_n}{b_0R^n}\right)\ge\frac{d_0}{2}.
\]
All earlier iterates on this event are also nonpositive. Consequently,
\begin{equation}\label{eq:hp-average}
 |\bar x_{T_n}|\ge\frac{T_n-n}{T_n}\frac{d_0}{2}\ge\frac{d_0}{4},
 \qquad f(\bar x_{T_n})-f^\star\ge\frac{\ell d_0^2}{32}=:c_\star.
\end{equation}
Taking $\delta_n=\pi_n/2$ proves the probability statement in \eqref{eq:hp-lower}. To compare confidence and horizon, observe that
\begin{align*}
 \log(1/\delta_n)
 &=\log2+\frac{n+1}{2}\log2+n\log(4/p_\star)+n(n-1)\log R=O(n^2),\\
 \log T_n&=\frac n2\log R+O(1).
\end{align*}
This gives the second statement in \eqref{eq:hp-lower} with a fixed parameter-dependent $K$. The sequence over sufficiently large odd $n$ is eventually strictly increasing; relabeling it yields the sequence asserted in the theorem.

Finally, for every fixed polynomial $P$ and $a,b>0$,
$T_n^{-a}+P(\log(1/\delta_n))T_n^{-b}\to0$.
Thus the right-hand side of \eqref{eq:hp-excluded} is eventually smaller than $c_\star$, while its failure probability is at least $2\delta_n>\delta_n$. This proves the claimed impossibility on one fixed bounded instance.

For clarity, the same rate formula and constants are required for every pair $(T,\delta)$, but the theorem does not require a single probability event holding simultaneously over all times. This quantitative requirement should not be confused with qualitative convergence at a fixed confidence level. In this construction the internal odd integer $N$ is finite almost surely, and noise is permanently switched off no later than time $N+1$. Thereafter $x_{t+1}=(1-\eta\ell/b_t)x_t$, with $0<\eta\ell/b_t\le1/2$. On each path, boundedness gives $b_t^2\le K(\omega)+4\ell^2d_0^2t$ for a finite path-dependent constant $K(\omega)$, so $\sum_t\eta\ell/b_t=\infty$. The contraction product tends to zero, hence $x_t\to0$ and $\bar x_T\to0$ almost surely. Bounded convergence also gives $\mathbb E[f(\bar x_T)-f^\star]\to0$. What fails is precisely the rate/confidence combination in \eqref{eq:hp-excluded}; the rare event probabilities are not bounded away from zero as $n\to\infty$.

\end{proof}

\section{Proof of Theorem~\ref{thm:no-bounded}}\label{app:no-bounded}
Appendix~\ref{app:hp} retained a uniform trajectory radius and isolated the dependence on confidence. Removing that radius allows a different obstruction: a horizon-dependent family with constant expected average error, even though its additive noise is independent and identically distributed on each chosen instance.

\begin{proof}
We use rare corrections to make a usually misleading gradient unbiased, while keeping the probability of no correction bounded below over the chosen horizon. Fix a sufficiently large integer $T\ge2$ and set
\begin{equation}\label{eq:variance-parameters}
 q=\frac1T,\qquad a=\frac{\sigma\sqrt q}{4},\qquad
 B_0=\sqrt{b_0^2+\sigma^2/16}.
\end{equation}
Take $T$ large enough that $a\le1$, and define
\begin{equation}\label{eq:variance-objective}
 f_T(x)=\begin{cases}
 0,&x\le0,\\
 x^2/2,&0<x<a,\\
 ax-a^2/2,&x\ge a.
 \end{cases}
\end{equation}
Its derivative is continuous, nondecreasing, and $1$-Lipschitz. Thus $f_T$ is convex and $1$-smooth, with $f_T^\star=0$. Choose $x^\star=0$ and initialize at $x_1=1$.

For this fixed instance, draw independent identically distributed noise at every time with law
\begin{equation}\label{eq:variance-noise}
 \xi_t=\begin{cases}
 -2a,&\text{with probability }1-q,\\
 2a(1-q)/q,&\text{with probability }q.
 \end{cases}
\end{equation}
The probabilities and values do not vary with time. Direct calculation gives
\[
 \mathbb E[\xi_t]=0,\qquad
 \mathbb E[\xi_t^2]=\frac{4a^2(1-q)}q\le\frac{\sigma^2}{4}.
\]
Therefore $g_t=f_T'(x_t)+\xi_t$ satisfies the conditional unbiasedness and variance assumptions for all $t\ge1$. The objective and noise law depend on the chosen horizon, but the variance bound is common to the entire family.

The noise law is unbiased only after averaging over its rare correction. To obtain a lower bound, we isolate the complementary event on which that correction never occurs. Let $\mathcal E_T$ be the event that the first $T$ noise values are all $-2a$. Then
$\mathbb P(\mathcal E_T)=(1-1/T)^T\ge1/4$.

On $\mathcal E_T$, induction shows that $x_t\ge1\ge a$ and $g_t=-a$ through time $T$. The updates thus remain in the linear region and move to the right. From \eqref{eq:variance-parameters}, for $t\le T$,
\begin{equation}\label{eq:variance-trajectory}
 b_t^2=b_0^2+ta^2\le B_0^2,\qquad
 x_t\ge1+\frac{\eta a}{B_0}(t-1),\qquad
 \bar x_T\ge1+\frac{\eta a(T-1)}{2B_0}.
\end{equation}
Because $a-a^2/2\ge0$, evaluating \eqref{eq:variance-objective} on this event gives
\begin{equation}\label{eq:variance-event-gap}
 f_T(\bar x_T)
 \ge\frac{\eta a^2(T-1)}{2B_0}
 =\frac{\eta\sigma^2(T-1)}{32B_0T}
 \ge\frac{\eta\sigma^2}{64B_0}.
\end{equation}
The objective is nonnegative on the complement of $\mathcal E_T$. Multiplying the last bound by $\mathbb P(\mathcal E_T)\ge1/4$ proves \eqref{eq:variance-worstcase}. This argument fixes an instance for each $T$; it does not produce nonconvergence of one instance as $T\to\infty$.

\end{proof}

\section{Proof of Theorem~\ref{thm:log-lb}}\label{app:log-lb}
The previous lower bound used rare large corrections under a variance constraint; we now ask what remains possible when the noise is uniformly bounded. A constant-magnitude scalar oracle exposes harmonic accumulation of fluctuations, yielding the logarithmic average-error lower bound for both algorithms.

\begin{proof}
Fix a sufficiently large horizon $T$. We choose a shallow linear region so that the restoring drift is smaller than the accumulated fluctuations, and then smooth its junction with the minimum region. Set $\gamma=\sigma/2$, choose a universal constant $c_0>0$ below, and define
\begin{equation}\label{eq:log-objective}
 a=c_0\gamma\sqrt{\frac{\log T}{T}},\qquad
 f_T(x)=\begin{cases}
 0,&x\le0,\\
 Lx^2/2,&0<x<a/L,\\
 ax-a^2/(2L),&x\ge a/L.
 \end{cases}
\end{equation}
The derivative is continuous, nondecreasing, and $L$-Lipschitz, so $f_T$ is convex and $L$-smooth. Its minimum is zero, and we initialize at $x_1=x^\star=0$.

For large enough $T$, $a\le\gamma/2$. Write $h_t=f_T'(x_t)\in[0,a]$ and, using fresh randomness at each query, let
\begin{equation}\label{eq:log-oracle}
 \mathbb P(g_t=\gamma\mid\mathcal F_{t-1})=\frac12(1+h_t/\gamma),\qquad
 \mathbb P(g_t=-\gamma\mid\mathcal F_{t-1})=\frac12(1-h_t/\gamma).
\end{equation}
These probabilities are valid and give $\mathbb E[g_t\mid\mathcal F_{t-1}]=h_t$. For $\xi_t=g_t-h_t$,
\begin{equation}\label{eq:log-noise}
 |\xi_t|\le\gamma+a\le\frac{3\sigma}{4},\qquad
 \mathbb E[\xi_t^2\mid\mathcal F_{t-1}]=\gamma^2-h_t^2\ge\frac34\gamma^2.
\end{equation}
The pathwise noise bound implies Assumption~\ref{ass:sub-Gaussian noise}. The oracle is defined for every time $t\ge1$, not just for the first $T$ queries. Moreover, $g_t^2=\gamma^2$ on every path, so both algorithms have the deterministic accumulator
\begin{equation}\label{eq:log-accumulator}
 b_t^2=b_0^2+t\gamma^2.
\end{equation}
For a fixed chosen $T$, the oracle depends on the current point through $h_t$, rather than on an externally specified time schedule.

With the oracle verified, the deterministic denominator allows us to expand the recursion with nonrandom weights. This is the step that isolates the source of the logarithmic factor. Define
\begin{equation}\label{eq:log-decomposition}
 \bar x_T=Y_T-\mathcal D_T,\qquad
 Y_T=-\sum_{s=1}^{T-1}\alpha_s\xi_s,\qquad
 \mathcal D_T=\sum_{s=1}^{T-1}\alpha_sh_s,\qquad
 \alpha_s=\frac{\eta(T-s)}{Tb_s}.
\end{equation}
Here $Y_T$ is a sum of martingale differences and has mean zero. The drift is random, but it admits the deterministic bound
\begin{equation}\label{eq:log-drift}
 0\le\mathcal D_T\le\frac{\eta a}{\gamma}\sum_{s=1}^{T-1}s^{-1/2}
 \le2c_0\eta\sqrt{\log T}=:d_T^{\max}.
\end{equation}
We use this pathwise upper bound when comparing positive parts; the random drift itself is never moved outside an expectation.

To quantify the fluctuations, put $W_T=\gamma^2\sum_{s=1}^{T-1}\alpha_s^2$. Orthogonality of martingale differences and \eqref{eq:log-noise} imply
\begin{equation}\label{eq:log-variance}
 \mathbb E[Y_T]=0,\qquad
 \frac34W_T\le\mathbb E[Y_T^2]\le W_T.
\end{equation}
Let $s_0=\max\{1,\lceil b_0^2/\gamma^2\rceil\}$. For $s_0\le s\le\lfloor T/2\rfloor$, we have $b_s^2\le2s\gamma^2$ and $(T-s)/T\ge1/2$. Consequently, for all sufficiently large $T$,
\begin{equation}\label{eq:log-harmonic}
 \frac{\eta^2}{16}\log T
 \le\frac{\eta^2}{8}\sum_{s=s_0}^{\lfloor T/2\rfloor}\frac1s
 \le W_T\le\eta^2\sum_{s=1}^{T-1}\frac1s
 \le\eta^2(1+\log T).
\end{equation}
The lower estimate holds once the finite initial segment up to $s_0$ is negligible relative to $\log T$.

The second-moment estimate must now be converted into a positive fluctuation. A variance lower bound alone would allow that fluctuation to be concentrated on very rare events. The bounded martingale increments give a fourth-moment bound that prevents this, and hence a lower bound on the expected positive part.

Set $X_s=-\alpha_s\xi_s$ and $d_s=3\gamma\alpha_s/2$. Then $\mathbb E[X_s\mid\mathcal F_{s-1}]=0$ and $|X_s|\le d_s$, with $d_s$ deterministic. Convexity of the exponential on $[-d_s,d_s]$ gives, for any $u\in\mathbb R$,
\[
 \mathbb E[e^{uX_s}\mid\mathcal F_{s-1}]
 \le\cosh(ud_s)\le e^{u^2d_s^2/2}.
\]
Iterated conditioning and Chernoff's inequality therefore yield
\[
 \mathbb P(|Y_T|\ge v)\le2e^{-v^2/(2V_T)},\qquad
 V_T=\sum_{s=1}^{T-1}d_s^2=\frac94W_T.
\]
Integrating the tail proves
\begin{equation}\label{eq:log-fourth}
 \mathbb E[Y_T^4]
 \le\int_0^\infty8v^3e^{-v^2/(2V_T)}\,dv
 =16V_T^2=81W_T^2.
\end{equation}
H\"older's inequality and \eqref{eq:log-variance} now give
\[
 \mathbb E[|Y_T|]
 \ge\frac{(\mathbb E[Y_T^2])^{3/2}}{(\mathbb E[Y_T^4])^{1/2}}
 \ge\frac{\sqrt3}{24}\sqrt{W_T}.
\]
Since $Y_T$ is integrable and mean zero, $\mathbb E[(Y_T)_+]=\mathbb E[|Y_T|]/2$; symmetry of its law is not needed. Combining this identity with \eqref{eq:log-harmonic}, we obtain
\begin{equation}\label{eq:log-positive}
 \mathbb E[(Y_T)_+]\ge k\eta\sqrt{\log T},\qquad k=\frac{\sqrt3}{192}.
\end{equation}

We can now convert this fluctuation into objective error. It remains to choose the slope so that neither the restoring drift nor the smoothing interval absorbs the positive fluctuation. Set $c_0=k/8$ and increase the threshold on $T$ so that $a/L\le(k/4)\eta\sqrt{\log T}$. Using $(y-z)_+\ge y_+-z$ for $z\ge0$, together with \eqref{eq:log-decomposition}, \eqref{eq:log-drift}, and \eqref{eq:log-positive}, gives
\begin{equation}\label{eq:log-smoothed-positive}
 \mathbb E\!\left[\left(\bar x_T-\frac aL\right)_+\right]
 \ge\mathbb E[(Y_T)_+]-d_T^{\max}-\frac aL
 \ge\frac k2\eta\sqrt{\log T}.
\end{equation}
For every $x\in\mathbb R$, \eqref{eq:log-objective} implies
$f_T(x)\ge a(x-a/L)_+$. On $x<a/L$ this follows from nonnegativity, and on $x\ge a/L$ it follows from the linear formula. Hence
\[
 \mathbb E[f_T(\bar x_T)-f_T^\star]
 \ge\frac{ak\eta}{2}\sqrt{\log T}
 =\frac{k^2}{32}\eta\sigma\frac{\log T}{\sqrt T}.
\]
This proves \eqref{eq:log-lower} with the universal constant $c=k^2/32$. Every lower threshold on $T$ depends only on $L,\sigma,b_0,\eta$. The same scalar function and oracle prove the statement for both algorithms.

\end{proof}

We now verify \eqref{eq:sgd-lower}, including the independent additive-noise version. This consequence is separate from the AdaGrad oracle construction: constant gradient magnitude is needed to make the adaptive denominator deterministic, but is unnecessary when SGD already uses deterministic stepsizes.
\begin{proof}[Proof of the SGD consequence]
Fix $L,\sigma,\lambda,\tau>0$, and put $\gamma=\sigma/2$, $\widetilde\eta=\lambda\gamma$, and $\widetilde b_0=\gamma\sqrt\tau$. Then
\[
 \frac{\widetilde\eta}{\sqrt{\widetilde b_0^2+t\gamma^2}}
 =\frac{\lambda}{\sqrt{t+\tau}}.
\]
With the constant-magnitude oracle \eqref{eq:log-oracle}, this identity makes the SGD and scalar AdaGrad trajectories exactly equal. Substituting $\eta=\widetilde\eta$ in \eqref{eq:log-lower} already proves the conditional-oracle version of \eqref{eq:sgd-lower}.

For the stronger independent-noise statement, use the same objective \eqref{eq:log-objective} with $a=c_0\gamma\sqrt{\log T/T}$, $c_0=k/8$, and $k=\sqrt3/192$, but draw independent Rademacher noise $\xi_t\in\{-\gamma,\gamma\}$. Let $g_t=f_T'(x_t)+\xi_t$. Conditional unbiasedness holds, and $\mathbb E[\exp(\xi_t^2/\sigma^2)\mid\mathcal F_{t-1}]=e^{1/4}\le e$. All these noise laws are fixed independently of the evaluation horizon.

Write $h_t=f_T'(x_t)\in[0,a]$ and
\[
 w_s=\frac{\lambda(T-s)}{T\sqrt{s+\tau}},\qquad
 Y_T=-\sum_{s<T}w_s\xi_s,\qquad
 \mathcal D_T=\sum_{s<T}w_sh_s,\qquad
 W_T=\gamma^2\sum_{s<T}w_s^2.
\]
The SGD recursion gives $\bar x_T=Y_T-\mathcal D_T$, and
\[
 0\le\mathcal D_T\le2c_0\widetilde\eta\sqrt{\log T},\qquad
 \mathbb E[Y_T]=0,\qquad \mathbb E[Y_T^2]=W_T.
\]
The harmonic estimates \eqref{eq:log-harmonic} apply with $\eta=\widetilde\eta$, $b_0=\widetilde b_0$, since $w_s$ is exactly the previous deterministic weight. Independence also gives the elementary fourth-moment identity
\[
 \mathbb E[Y_T^4]
 =\gamma^4\sum_{s<T}w_s^4
   +6\gamma^4\sum_{r<s<T}w_r^2w_s^2
 =3W_T^2-2\gamma^4\sum_{s<T}w_s^4\le3W_T^2.
\]
Thus the weaker estimates $\mathbb E[Y_T^2]\ge3W_T/4$ and $\mathbb E[Y_T^4]\le81W_T^2$ used in \eqref{eq:log-positive} remain valid. They yield
$\mathbb E[(Y_T)_+]\ge k\widetilde\eta\sqrt{\log T}$.
The same drift and smoothing comparison \eqref{eq:log-smoothed-positive} therefore gives
\[
 \mathbb E[f_T(\bar x_T)-f_T^\star]
 \ge\frac{k^2}{32}\widetilde\eta\sigma\frac{\log T}{\sqrt T}
 =\frac{k^2}{64}\lambda\sigma^2\frac{\log T}{\sqrt T}.
\]
This proves \eqref{eq:sgd-lower} with $c_{\mathrm{SGD}}=k^2/64$. Every sufficiently-large-$T$ threshold depends only on $L,\sigma,\lambda,\tau$. Only the hard objective changes with $T$; the stepsize schedule and the independent noise distribution do not.
\end{proof}

\paragraph{The finite-horizon comparison.}
For completeness, let SGD use a constant $0<\alpha\le1/(2L)$ and put $z_t=x_t-x^\star$. Conditional unbiasedness, convexity, and the smooth self-bound give
\[
 \mathbb E[\|z_{t+1}\|^2\mid\mathcal F_{t-1}]
 \le\|z_t\|^2-2\alpha(1-L\alpha)\Delta_t+\alpha^2\sigma^2
 \le\|z_t\|^2-\alpha\Delta_t+\alpha^2\sigma^2.
\]
All finite-time moments here exist: $\|\nabla f(x_t)\|\le L\|z_t\|$ and the variance bound imply, recursively, a finite second moment for $z_t$ from deterministic $x_1$. Taking expectations, summing over $t\le T$, dropping the nonnegative terminal squared distance, and using convexity of $f$ proves \eqref{eq:sgd-fixed-horizon}. This comparison is a direct distance-recursion calculation, not a claim that every SGD schedule avoids logarithmic losses.

\section{Proof of Theorem~\ref{thm:norm upper bound}}\label{app:positive}
The lower bounds above isolate failures that averaging, tail assumptions, and stability do not individually remove. We now combine expected averaging with bounded iterates and the conditional ABC condition. The proof follows the classical regret and self-bounding approach \citep{duchi2011adaptive,levy2018online,vaswani2020adaptive}, while retaining the initialization and noise coefficients. After establishing the shared estimates, we give a separate proof for each algorithm.

\subsection{Shared self-bounding and integrability estimates}
Both proofs close the same quadratic inequality, but use different pathwise regret bounds. We first record the common second-moment estimate and justify the expectations used later. Smoothness applied at $x-\nabla f(x)/L$ gives
\[
 f^\star\le f\!\left(x-\frac{\nabla f(x)}L\right)
 \le f(x)-\frac{\|\nabla f(x)\|^2}{2L}.
\]
Therefore
\begin{equation}\label{eq:self-bound}
 \|\nabla f(x)\|^2\le2L(f(x)-f^\star).
\end{equation}
Writing $S_T=\sum_{t=1}^T\mathbb E[\Delta_t]$ and using $Q=A+2LB$, the ABC condition implies
\begin{equation}\label{eq:ABC-sum}
 \sum_{t=1}^T\mathbb E[\|g_t\|^2]\le QS_T+CT.
\end{equation}
All terms are finite. Indeed, smoothness at $x^\star$, where $\nabla f(x^\star)=0$, bounds $\Delta_t$ by $LD^2/2$ for AdaGrad-Norm and by $(L/2)\sum_iD_i^2$ for AdaGrad. The ABC condition then gives uniform conditional second-moment bounds. In particular, the accumulators have finite first moments at each finite time, and the inner products with bounded iterates are integrable. These observations justify both the martingale cancellation and the applications of Jensen's inequality below.

\subsection{AdaGrad-Norm}
Scalar normalization allows a single Euclidean distance identity. The important order of operations is to sum this identity pathwise before taking expectations, so the correlation between $g_t$ and $b_t$ never needs to be discarded.

\begin{proof}[Proof of Theorem~\ref{thm:norm upper bound}, part~1]
Write $z_t=x_t-x^\star$. The actual current-gradient update gives
\[
 \langle g_t,z_t\rangle
 =\frac{b_t}{2\eta}(\|z_t\|^2-\|z_{t+1}\|^2)
 +\frac{\eta\|g_t\|^2}{2b_t}.
\]
Summation by parts and $\|z_t\|\le D$ yield
\begin{align*}
 \sum_{t=1}^Tb_t(\|z_t\|^2-\|z_{t+1}\|^2)
 &=b_1\|z_1\|^2+\sum_{t=2}^T(b_t-b_{t-1})\|z_t\|^2-b_T\|z_{T+1}\|^2\\
 &\le D^2b_T.
\end{align*}
Also,
\[
 \frac{\|g_t\|^2}{b_t}
 =\frac{(b_t-b_{t-1})(b_t+b_{t-1})}{b_t}
 \le2(b_t-b_{t-1}).
\]
Consequently,
\begin{equation}\label{eq:norm-pathwise}
 \sum_{t=1}^T\langle g_t,z_t\rangle
 \le\left(\frac{D^2}{2\eta}+\eta\right)b_T-\eta b_0
 \le\Gamma b_T.
\end{equation}
Because $z_t$ is bounded and $\mathcal F_{t-1}$-measurable, conditional unbiasedness gives $\mathbb E[\langle\xi_t,z_t\rangle]=0$. Convexity and \eqref{eq:norm-pathwise} thus imply
\[
 S_T\le\sum_{t=1}^T\mathbb E[\langle\nabla f(x_t),z_t\rangle]
 =\sum_{t=1}^T\mathbb E[\langle g_t,z_t\rangle]
 \le\Gamma\mathbb E[b_T].
\]
Applying Jensen's inequality only now, and then \eqref{eq:ABC-sum}, gives
\begin{equation}\label{eq:norm-self-consistency}
 S_T\le\Gamma\sqrt{\mathbb E[b_T^2]}
 \le\Gamma\sqrt{b_0^2+QS_T+CT}.
\end{equation}
Since $S_T\ge0$, squaring and solving the quadratic inequality yields
\[
 S_T\le\frac{\Gamma^2Q+\sqrt{\Gamma^4Q^2+4\Gamma^2(b_0^2+CT)}}2.
\]
Convexity gives $\mathbb E[f(\bar x_T)-f^\star]\le S_T/T$, hence the sharper estimate
\begin{equation}\label{eq:norm-adaptive-root}
 \mathbb E[f(\bar x_T)-f^\star]
 \le\frac{\Gamma^2Q+\sqrt{\Gamma^4Q^2+4\Gamma^2(b_0^2+CT)}}{2T}.
\end{equation}
Finally,
\[
 \sqrt{\Gamma^4Q^2+4\Gamma^2(b_0^2+CT)}
 \le\Gamma^2Q+2\Gamma\sqrt{b_0^2+CT}
 \le\Gamma^2Q+2\Gamma b_0+2\Gamma\sqrt{CT},
\]
which proves \eqref{eq:norm-noise-adaptive} with the stated constants.
\end{proof}

\subsection{Coordinate-wise AdaGrad}
Each coordinate has its own accumulator, so one first bounds its scalar regret separately. Summing these bounds and applying Cauchy--Schwarz recovers a scalar quadratic inequality without requiring coordinate-wise ABC assumptions.

\begin{proof}[Proof of Theorem~\ref{thm:norm upper bound}, part~2]
Let $z_{t,i}=x_{t,i}-x_i^\star$. The update in coordinate $i$ gives
\[
 g_{t,i}z_{t,i}
 =\frac{b_{t,i}}{2\eta}(z_{t,i}^2-z_{t+1,i}^2)
 +\frac{\eta g_{t,i}^2}{2b_{t,i}}.
\]
Since $|z_{t,i}|\le D_i$ and $b_{t,i}$ is nondecreasing,
\[
 \sum_{t=1}^Tb_{t,i}(z_{t,i}^2-z_{t+1,i}^2)\le D_i^2b_{T,i},
 \qquad
 \sum_{t=1}^T\frac{g_{t,i}^2}{b_{t,i}}\le2(b_{T,i}-b_{0,i}).
\]
It follows pathwise that
\begin{equation}\label{eq:coordinate-pathwise}
 \sum_{t=1}^Tg_{t,i}z_{t,i}
 \le\left(\frac{D_i^2}{2\eta}+\eta\right)b_{T,i}-\eta b_{0,i}
 \le\Gamma_i b_{T,i}.
\end{equation}
Summing over $i$, taking expectations, and using convexity and conditional unbiasedness after this pathwise estimate gives
\begin{equation}\label{eq:coordinate-sum}
 S_T\le\sum_{i=1}^d\Gamma_i\mathbb E[b_{T,i}]
 \le\sum_{i=1}^d\Gamma_i
 \sqrt{b_{0,i}^2+\sum_{t=1}^T\mathbb E[g_{t,i}^2]}.
\end{equation}
The martingale terms cancel because $z_{t,i}$ is bounded and $\mathcal F_{t-1}$-measurable. Cauchy--Schwarz and the global estimate \eqref{eq:ABC-sum} now yield
\begin{equation}\label{eq:coordinate-self-consistency}
 \begin{aligned}
 S_T
 &\le\left(\sum_{i=1}^d\Gamma_i^2\right)^{1/2}
 \left(\sum_{i=1}^db_{0,i}^2+\sum_{t=1}^T\mathbb E[\|g_t\|^2]\right)^{1/2}\\
 &\le\Lambda\sqrt{\beta_0^2+QS_T+CT}.
 \end{aligned}
\end{equation}
Squaring and taking the positive root gives
\[
 S_T\le\frac{\Lambda^2Q+\sqrt{\Lambda^4Q^2+4\Lambda^2(\beta_0^2+CT)}}2.
\]
Dividing by $T$ and using convexity of $f$ gives the sharper estimate
\begin{equation}\label{eq:coord-adaptive-root}
 \mathbb E[f(\bar x_T)-f^\star]
 \le\frac{\Lambda^2Q+\sqrt{\Lambda^4Q^2+4\Lambda^2(\beta_0^2+CT)}}{2T}.
\end{equation}
The estimate
\[
 \sqrt{\Lambda^4Q^2+4\Lambda^2(\beta_0^2+CT)}
 \le\Lambda^2Q+2\Lambda\beta_0+2\Lambda\sqrt{CT}
\]
then gives \eqref{eq:coordinate-noise-adaptive}. This argument uses only the global ABC condition. The intermediate bound \eqref{eq:coordinate-sum} preserves more detailed coordinate-wise second-moment information when it is available.
\end{proof}

\subsection{Noise specializations and the role of stability}
The two proofs have the same final dependence on the residual coefficient $C$, even though their geometric constants differ. When $C=0$, their simplified bounds are $(\Gamma^2Q+\Gamma b_0)/T$ and $(\Lambda^2Q+\Lambda\beta_0)/T$. This case allows nonzero state-dependent noise; it excludes only an additive residual term in the ABC model.

Under conditional unbiasedness and bounded variance,
\[
 \mathbb E[\|g_t\|^2\mid\mathcal F_{t-1}]
 =\|\nabla f(x_t)\|^2+\mathbb E[\|\xi_t\|^2\mid\mathcal F_{t-1}]
 \le\|\nabla f(x_t)\|^2+\sigma^2.
\]
Thus $(A,B,C)=(0,1,\sigma^2)$, so $Q=2L$ and
\[
 \mathbb E[f(\bar x_T)-f^\star]
 \le\frac{2L\Gamma^2+\Gamma b_0}{T}+\frac{\Gamma\sigma}{\sqrt T}
\]
for AdaGrad-Norm, with the analogous substitution $(\Gamma,b_0)\mapsto(\Lambda,\beta_0)$ for AdaGrad. These guarantees assume a common almost-sure stability radius. They therefore do not contradict the horizon-dependent lower bounds without such a radius, nor do they imply a last-iterate rate or polynomial high-probability bounds with polylogarithmic confidence dependence. Markov's inequality does give the weaker $1/\delta$ dependence in \eqref{eq:markov-average}.

\endgroup
\end{document}